\documentclass[11pt]{amsart}

\usepackage{amscd, verbatim}
\usepackage{amsmath, amssymb}
\usepackage{amsfonts}

\newcommand{\de}{\partial}
\newcommand{\db}{\overline{\partial}}

\newcommand{\ddbar}{\sqrt{-1} \partial \overline{\partial}}

\newcommand{\Ric}{\mathrm{Ric}}

\newcommand{\ov}[1]{\overline{#1}}

\newcommand{\mn}{\sqrt{-1}}

\newcommand{\tr}[2]{\textrm{tr}_{#1}{#2}}
\newcommand{\ti}[1]{\tilde{#1}}
\newcommand{\vp}{\varphi}

\newcommand{\hol}{\textrm{Hol}}

\renewcommand{\leq}{\leqslant}
\renewcommand{\geq}{\geqslant}

\begin{document}
\newtheorem{claim}{Claim}
\newtheorem{theorem}{Theorem}[section]
\newtheorem{lemma}[theorem]{Lemma}
\newtheorem{corollary}[theorem]{Corollary}
\newtheorem{proposition}[theorem]{Proposition}
\newtheorem{conj}[theorem]{Conjecture}
\newtheorem{question}{question}[section]

\theoremstyle{definition}
\newtheorem{example}[theorem]{Example}
\newtheorem{remark}[theorem]{Remark}

\title{Non-K\"ahler Calabi-Yau manifolds}
\author{Valentino Tosatti}
\thanks{Research supported in part by a Sloan Research Fellowship and NSF grant DMS-1308988.}
\address{Department of Mathematics, Northwestern University, 2033 Sheridan Road, Evanston, IL 60208}
\email{tosatti@math.northwestern.edu}
\dedicatory{Dedicated to Professor Duong H. Phong on the occasion of his 60th birthday}

\begin{abstract}
We study the class of compact complex manifolds whose first Chern class vanishes in the Bott-Chern cohomology.
This class includes all manifolds with torsion canonical bundle, but it is strictly larger.
After making some elementary remarks, we show that a manifold in Fujiki's class $\mathcal{C}$ with
vanishing first Bott-Chern class has torsion canonical bundle.
We also give some examples of non-K\"ahler Calabi-Yau manifolds, and discuss the problem of defining and
constructing canonical metrics on them.
\end{abstract}

\maketitle

\section{Introduction}
In this paper, Calabi-Yau manifolds are defined to be compact K\"ahler manifolds $M$ with $c_1(M)=0$ in $H^2(M,\mathbb{R})$. Thanks to Yau's theorem \cite{Ya1} these are precisely the compact manifolds that admit Ricci-flat K\"ahler metrics.
Using this, Calabi proved a decomposition theorem \cite{Ca} which shows that any such manifold has a finite unramified cover which splits as a product of a torus and a Calabi-Yau manifold with vanishing first Betti number. From this one can easily deduce that Calabi-Yau manifolds have holomorphically torsion canonical bundle (see Theorem \ref{old}).

One can ask how much of this theory carries over to the case of non-K\"ahler Hermitian manifolds.
Simple examples, such as a Hopf surface diffeomorphic to $S^1\times S^3$, show that the condition that $c_1(M)=0$ in $H^2(M,\mathbb{R})$ is too weak in general (see Example \ref{exam2}). On the other hand, much interest has been devoted to studying non-K\"ahler compact complex manifolds with holomorphically trivial (or more generally torsion) canonical bundle, and many examples can be found in \cite{BDV,COP,Cl,FP,Fr,Gh,GP,Gra,GGP,Gu, Hi,LT,Na,TY,Wi,Yo} and references therein. For example, every compact complex nilmanifold with a left-invariant complex structure has trivial canonical bundle, and it is always non-K\"ahler unless it is a torus \cite{BDV}.
A lot of interest in the subject was generated by ``Reid's fantasy'' \cite{Re} that all Calabi-Yau threefolds with trivial canonical bundle should form a connected family provided one allows deformations and singular transitions through non-K\"ahler manifolds with trivial canonical bundle. The geometry of compact complex manifolds with trivial canonical bundle has been investigated for example by \cite{AI, BDV, Boz,DLV, FG, F, FLY, FWW, FY, Gh, IP, LY, Po, Ro} and others.
In this paper we will consider a more general class of manifolds, that we now define, and argue that they can naturally be considered as ``non-K\"ahler Calabi-Yau'' manifolds.

On any compact complex manifold there is a (finite-dimensional) cohomology theory called Bott-Chern cohomology. We will need only the real $(1,1)$ Bott-Chern cohomology
$$H^{1,1}_{\mathrm{BC}}(M, \mathbb{R})=\frac{\{d\textrm{-closed real }(1,1)\textrm{-forms}\}}{\{\ddbar \psi, \psi\in C^\infty(M,\mathbb{R})\}}.$$
There is a ``first Bott-Chern class'' map $c_1^{\mathrm{BC}}:\mathrm{Pic}(M)\to H^{1,1}_{\mathrm{BC}}(M,\mathbb{R}),$
which can be described as follows. Given any holomorphic line bundle $L\to M$ and any Hermitian metric $h$ on the fibers of $L$, its curvature form $R_h$ is locally given by $-\mn\de\db\log h$. Then $R_h$ is a closed real $(1,1)$-form and if we choose a different metric $h'$ then
$R_{h'}-R_h=\mn\de\db\log (h/h')$ is globally $\de\db$-exact, so we can defined $c_1^{\mathrm{BC}}(L)$ to be the class of $R_h$ in $H^{1,1}_{\mathrm{BC}}(M,\mathbb{R})$.

If $g$ is any Hermitian metric on $M$, with fundamental $2$-form $\omega$, then its first Chern form given locally by
$$\Ric(\omega) = - \mn\de\db \log \det g,$$
represents $c_1^{\mathrm{BC}}(K_M^*)=c_1^{\mathrm{BC}}(M)$. We will call $\Ric(\omega)$ the Chern-Ricci form of $\omega$.

We then define a non-K\"ahler Calabi-Yau manifold to be a compact complex manifold $M$ with $c_1^{\mathrm{BC}}(M)=0$ in $H^{1,1}_{\mathrm{BC}}(M,\mathbb{R})$. This class of manifolds is contained in the class
of compact complex manifolds with $c_1(M)=0$ in $H^2(M,\mathbb{R})$, and contains the class of compact complex manifolds with holomorphically torsion canonical bundle, but both inclusions are strict, see the examples
in section \ref{sectex}.

In this paper we investigate the structure of such manifolds.
The condition $c_1^{\mathrm{BC}}(M)=0$ simply means that given any Hermitian metric $\omega$ on $M$, its Ricci form satisfies $\Ric(\omega)=\mn\de\db F$ for some $F\in C^\infty(M,\mathbb{R})$.
One can think of these manifolds as possibly non-K\"ahler Calabi-Yau manifolds. In fact, the conformally rescaled metric $e^{F/n}\omega$ has vanishing Chern-Ricci curvature. More interestingly, there are other quite different ways of constructing Chern-Ricci flat metrics on such manifolds, see Theorem \ref{metrics} below. Furthermore, as shown by Gill \cite{Gi}, on such manifolds the Chern-Ricci flow \cite{TW2,TW3,TWY} deforms any given Hermitian metric to a Chern-Ricci flat one.

Let us start with some elementary characterizations of this class of manifolds.

\begin{proposition}\label{prop} Let $M$ be a compact complex manifold with $n=\dim_\mathbb{C} M$. The following are equivalent:
\begin{enumerate}
\item $c_1^{\mathrm{BC}}(M)=0$ in $H^{1,1}_{\mathrm{BC}}(M,\mathbb{R})$
\item There exists a Hermitian metric $\omega$ on $M$ with $\Ric(\omega)=0$
\item There exists a Hermitian metric on the fibers of $K_M$ with vanishing curvature
(i.e. $K_M$ is unitary flat)
\item There exists a Hermitian metric $\omega$ on $M$ whose Chern connection
has restricted holonomy contained in $SU(n)$
\end{enumerate}
Lastly, if $K_M$ is holomorphically torsion (i.e. there exists $\ell\geq 1$ such that
$K_M^{\ell}\cong \mathcal{O}_M$) then $c_1^{\mathrm{BC}}(M)=0$.
\end{proposition}

Next, we make part (2) in the above proposition more precise:

\begin{theorem}\label{metrics}
Let $M$ be a compact complex manifold with $c_1^{\mathrm{BC}}(M)=0$. Given Hermitian metrics $\omega,\omega_0$ on $M$, we can find
Hermitian metrics $\tilde{\omega}_j$ on $M$, $j=1,2,3$, with $\Ric(\ti{\omega}_j)=0$ which are given by the following forms:
\begin{enumerate}
\item $\ti{\omega}_1=e^{\vp_1}\omega,$
\item $\ti{\omega}_2=\omega+\ddbar\vp_2,$
\item $\ti{\omega}_3^{n-1}=\omega^{n-1}+\ddbar\vp_3\wedge\omega_0^{n-2},$
\end{enumerate}
where $\vp_1,\vp_2,\vp_3\in C^\infty(M,\mathbb{R})$ are unique up to addition of a constant.
\end{theorem}

While part (1) of this result is completely elementary, parts (2) and (3) are considerably harder. When $\omega$ is K\"ahler, (2) is Yau's solution of the Calabi Conjecture \cite{Ya1}, and in the non-K\"ahler case (2) follows from work of Cherrier \cite{Ch} when $n=2$ and of Weinkove and the author \cite{TW} in general (see also \cite{GL, TW1}). Part (3) uses the solvability of an equation introduced by Fu-Wang-Wu \cite{FWW, FWW2}, which was recently established by Weinkove and the author \cite{TW4, TW5}.

Conditions (2) and (3) are more satisfactory than condition (1), since special properties of $\omega$ carry over to $\ti{\omega}$. For example, if $d\omega=0$ or $\de\db\omega=0$, then the same will be true for $\ti{\omega}_2$ constructed as in (2). Also, if $d\omega_0=0$ and $d(\omega^{n-1})=0$ (i.e. $\omega$ is balanced \cite{Mi}), then $\ti{\omega}_3$ as in (3) is balanced too. In the same setup, if $\de\db(\omega^{n-1})=0$ (i.e. $\omega$ is Gauduchon), then so is $\ti{\omega}_3$, and similarly if $\db(\omega^{n-1})$ is $\de$-exact (i.e. $\omega$ is strongly Gauduchon \cite{Po2}).

If one replaces restricted holonomy with its unrestricted version in Proposition \ref{prop}, then one gets the following:

\begin{proposition}\label{prop2}
Let $M$ be a compact complex manifold with $n=\dim_\mathbb{C} M$. The following are equivalent:
\begin{enumerate}
\item There exists a Hermitian metric $\omega$ on $M$ whose Chern connection
has (unrestricted) holonomy contained in $SU(n)$
\item The canonical bundle of $M$ is holomorphically trivial
\end{enumerate}
\end{proposition}

One is led to wonder whether every compact complex manifold with $c_1^{\mathrm{BC}}(M)=0$ has torsion canonical bundle.
This is not the case, as we will explain below in Examples \ref{exam} and \ref{exam4}. Nevertheless, $K_M$ is torsion in certain important cases.
To start with we have the following results, which are either elementary or already known.

\begin{theorem}\label{old} Let $M$ be a compact complex manifold with $c_1^{\mathrm{BC}}(M)=0$.
Then $K_M$ is holomorphically torsion provided any of the following conditions hold:
\begin{enumerate}
\item[(a)] $\kappa(M)\geq 0$, or more generally $H^0(M,K_M^{\ell})\neq 0$ for some $0\neq \ell\in\mathbb{Z}$,
\item[(b)] $b_1(M)=0$,
\item[(c)] $M$ is K\"ahler,
\item[(d)] $n=2$.
\end{enumerate}
\end{theorem}

Case (c), when $M$ is K\"ahler, was proved by Calabi \cite[Theorem 2]{Ca} assuming his famous conjecture, later proved by Yau \cite{Ya1}. Before Yau's work, Matsushima \cite[Theorem 3]{Ma} proved this result for algebraic manifolds and independently Bogomolov \cite[Theorem 3]{Bo}, Fujiki \cite[Proposition 6.6]{Fu} and Lieberman \cite[Theorem 3.13]{Li} proved it in the K\"ahler case.

Recall that a compact complex manifold is said to be in Fujiki's class $\mathcal{C}$ if it is bimeromorphic to a K\"ahler manifold (this was not Fujiki's original definition, but it is equivalent to it thanks to work of Varouchas \cite{Va}). Class $\mathcal{C}$ includes all Moishezon manifolds, which are bimeromorphic to projective manifolds. All manifolds in $\mathcal{C}$ satisfy the $\de\db$-Lemma and also have the property that holomorphic forms are closed (see \cite[Theorem 5.22]{DGMS} or \cite[Corollary 1.7]{Fu}). In particular if $M$ is in $\mathcal{C}$ then $c_1^{\mathrm{BC}}(M)=0$ if and only if $c_1(M)=0$ in $H^2(M,\mathbb{R})$. There are many examples of non-K\"ahler manifolds in class $\mathcal{C}$ with vanishing first Chern class. One class of such manifolds can be obtained as small resolutions of threefolds with only ordinary double points singularities and trivial canonical divisor (see e.g. \cite{Hi} or \cite{Fr2}). Another class of examples is obtained by applying a Mukai flop to a hyperk\"ahler manifold (see e.g. \cite[Example 21.7]{GHJ}, \cite[Section 4.4]{Yo} and also \cite{Gu}). We have the following result:

\begin{theorem}\label{main}
Let $M$ be a compact complex manifold which is in class $\mathcal{C}$ with $c_1^{\mathrm{BC}}(M)=0$ (equivalently, $c_1(M)=0$ in $H^2(M,\mathbb{R})$).
Then $K_M$ is holomorphically torsion.
\end{theorem}

Examples \ref{exam} and \ref{exam4} shows that for a general compact complex manifold, $c_1^{\mathrm{BC}}(M)=0$ does not imply that $K_M$ is torsion, and so there is no analog of Theorem \ref{main} outside of class $\mathcal{C}$. As a side remark, we note that Theorem \ref{main} was claimed in \cite[Theorem 8]{GIP}, but the proof given there is incorrect, since it does not distinguish between restricted and full holonomy (compare Propositions \ref{prop} and \ref{prop2}).

Finally, we study the invariance of the property $c_1^{\mathrm{BC}}(M)=0$ under small deformations of the complex structure. It is a classical fact that if $M$ is a compact K\"ahler manifold with $K_M$ torsion, then every sufficiently small deformation $M_t$ is also K\"ahler with $K_{M_t}$ torsion. It is also known that this fails if $M$ is not K\"ahler: for example Nakamura \cite{Na} constructed a complex parallelizable solvmanifold $M$ (so $K_M$ is trivial) which has arbitrarily small deformations $M_t$ with negative Kodaira dimension (so $K_{M_t}$ is not torsion), see Example \ref{exam4}.
We do not know whether the condition $c_1^{\mathrm{BC}}(M)=0$ is preserved by small deformations of the complex structure, but we have the following result:

\begin{proposition}\label{deform}
Let $M$ be a compact complex manifold with $c_1^{\mathrm{BC}}(M)=0$ and $b_1(M)=2h^{0,1}(M)$
(or equivalently, the $\de\db$-Lemma holds for $(1,1)$-forms). Then every sufficiently small deformation $M_t$ of
$M$ still satisfies $c_1^{\mathrm{BC}}(M_t)=0$ and $b_1(M_t)=2h^{0,1}(M_t)$.
\end{proposition}

This result applies for example to all small deformations of the Iwasawa threefold, studied in \cite{Na}, and shows that all these manifolds have vanishing first Bott-Chern class.
Indeed, the Iwasawa threefold $M$ is complex parallelizable (so in particular $K_M$ is trivial), and has $b_1(M)=4$ and $h^{0,1}(M)=2$ (see e.g. \cite[p.96]{Na}), and it follows that every sufficiently
small deformation $M_t$ satisfies $c_1^{\mathrm{BC}}(M_t)=0$. Note that Nakamura proved that many of these deformations are not complex parallelizable.\\

This paper is organized as follows. In section \ref{sectmain} we prove all the results stated in the Introduction. Section \ref{sectex} is devoted to several interesting examples of non-K\"ahler compact complex manifolds. Lastly in section \ref{sectbal} we discuss the problem of defining and constructing canonical metrics on non-K\"ahler Calabi-Yau manifolds.\\

\noindent {\bf Acknowledgements:} Part of this work was carried out
while the author was visiting the Mathematical Science
Center of Tsinghua University in Beijing, which he would like to
thank for the hospitality. He is grateful to V. Apostolov, G. Grantcharov, Z. Lu, G. Magnusson, D. Panov, M. Popa, S. Rollenske, A. Tomassini, B. Weinkove and X. Yang for useful discussions, and to the referee for
suggesting Example \ref{exam5}.

The author wishes to thank Professor Duong H. Phong for his invaluable help, support and guidance. {\em Buon Compleanno Phong!}
\section{The canonical bundle of non-K\"ahler Calabi-Yau manifolds}\label{sectmain}

In the following $\nabla$ will denote the Chern connection of a Hermitian metric and $\hol(\nabla)$, $\hol^0(\nabla)$ will be the unrestricted and restricted holonomy of $\nabla$. We have that  $\hol^0(\nabla)$ is the connected component of
the identity of  $\hol(\nabla)$. Throughout this paper we will assume that $n=\dim_{\mathbb{C}} M \geq 2$, since when $n=1$ all the results are trivial.

\begin{proof}[Proof of Proposition \ref{prop}]
By definition the condition that $c_1^{\mathrm{BC}}(M)=0$ means that given
any Hermitian metric $\omega$ on $M$ its Ricci form is $\de\db$-exact,
$\Ric(\omega)=\ddbar F$ for some smooth real function $F$.
So obviously $(2)$ implies $(1)$, and for the converse it is enough to consider
the Hermitian metric $e^{F/n}\omega$, whose Ricci form is $0$.
We now show that $(2), (3)$ and $(4)$ are equivalent (see \cite[Proposition 6.1.1]{Jo} for the K\"ahler case). First note that given any Hermitian metric $\omega$ its Chern connection $\nabla$ satisfies $\nabla J=0$, and therefore its holonomy $\hol(\nabla)$ is contained in $U(n)$.
Furthermore, $\nabla$ induces a connection $\nabla^K$ on the canonical bundle $K_M$ with
$\hol^0(\nabla^K)=\det\hol^0(\nabla)$ where $\det:U(n)\to U(1)$. It follows that
$\hol^0(\nabla)\subset SU(n)$ if and only if $\hol^0(\nabla^K)$ is trivial, which is equivalent to $\nabla^K$ being flat. But the curvature of $\nabla^K$ is exactly equal to $\Ric(\omega)$, and so $(2)$ is equivalent to $(3)$ and to $(4)$.

Finally, it is obvious that $K_M$ being torsion implies $(1)$ since $\ell c_1^{\mathrm{BC}}(M)=c_1^{\mathrm{BC}}(K_M^{\ell})=0$.
\end{proof}

\begin{proof}[Proof of Theorem \ref{metrics}]
Write again $\Ric(\omega)=\ddbar F$ for some smooth real function $F$. As noted above, for (a) it is enough to take $\vp_1=F/n$. Furthermore if $\Ric(e^{\ti{\vp}_1}\omega)=0$, then $\ddbar(\vp_1-\ti{\vp}_1)=0$ and so
$\ti{\vp}_1$ equals $\vp_1$ plus a constant.

For (b), we use the solution of the Hermitian complex Monge-Amp\`ere equation due to Cherrier when $n=2$ and Weinkove and the author \cite{TW} in general (see also \cite{GL, TW1}) to find a unique Hermitian metric of the form $\omega+\ddbar\vp_2$ which satisfies
$$(\omega+\ddbar\vp_2)^n=e^{F+b}\omega^n,$$
where $b$ is a uniquely determined real constant. Taking $\ddbar\log$ of this equation shows that $\Ric(\omega+\ddbar\vp_2)=0$, and again it is clear that $\vp_2$ is unique up to addition of a constant.

Lastly for (c), we use the recent solution of the Monge-Amp\`ere equation for $(n-1)$-plurisubharmonic functions by Weinkove and the author \cite[Theorem 1.1]{TW5} (see also our earlier work \cite{TW4} for the case when $\omega$ is K\"ahler, as well as \cite{FWW,FWW2}) to find a unique Hermitian metric $\ti{\omega}$ with
$\ti{\omega}^{n-1}=\omega^{n-1}+\ddbar\vp_3\wedge\omega_0^{n-2},$ solving the equation
$$\ti{\omega}^n=e^{F+b}\omega^n,$$
where $b$ is a uniquely determined real constant. Taking $\ddbar\log$ of this equation shows that $\Ric(\omega+\ddbar\vp_3)=0$, and the fact that $\vp_3$ is unique up to addition of a constant is proved in \cite{TW5}.
\end{proof}

The following lemma is well-known (cf. \cite{Ga}):
\begin{lemma}\label{parall}
Let $M$ be a compact complex manifold with $c_1^{\mathrm{BC}}(M)=0$ and fix a Hermitian metric $\omega$ with $\Ric(\omega)=0$. Then every holomorphic section of $K_M^\ell$, $\ell\in\mathbb{Z}$, is parallel with respect to $\nabla$.
\end{lemma}
\begin{proof}
If $\eta\in H^0(M,K_M^\ell)$ for some $\ell\in\mathbb{Z}$, and $|\eta|^2$ is its pointwise norm squared with respect to $\omega$ then a simple calculation shows that
\begin{equation}\label{impor}
\Delta |\eta|^2 = |\nabla \eta|^2+\ell g^{i\ov{j}}R_{i\ov{j}}|\eta|^2=|\nabla \eta|^2\geq 0,
\end{equation}
where $\Delta f = \tr{\omega}({\mn\de\db f})=g^{i\ov{j}}\de_i \de_{\ov{j}} f$ is the complex Laplacian of $\omega$.
Since $M$ is compact, the maximum principle implies that $|\eta|^2$ is constant and therefore $\nabla \eta=0$.
\end{proof}

\begin{proof}[Proof of Proposition \ref{prop2}]
Recall that the Chern connection $\nabla$ of any Hermitian metric has holonomy $\hol(\nabla)$ contained in $U(n)$.
We furthermore have that $\hol(\nabla)\subset SU(n)$ if and only if there exists a nontrivial $\nabla$-parallel $(n,0)$-form $\eta$ on $M$ (which necessarily has no zeros). Such a parallel $(n,0)$-form must be holomorphic, because
$\db\eta$ is the skewsymmetrization of $\nabla^{0,1}\eta$ (viewed as an element of $\Lambda^{0,1}\otimes\Lambda^{n,0}$) which is zero. Therefore $(1)$ implies $(2)$. For the converse, if $K_M$ is trivial then there exist on $M$ a never vanishing holomorphic $(n,0)$-form $\eta$ and a Hermitian metric $\omega$ with $\Ric(\omega)=0$. Lemma \ref{parall} shows that $\eta$ is parallel and therefore $\hol(\nabla)\subset SU(n)$.
\end{proof}

\begin{proof}[Proof of Theorem \ref{old}]

(a) We know from Lemma \ref{parall} that every holomorphic section of $K_M^{ \ell}$ (any $0\neq \ell\in\mathbb{Z}$) is parallel with respect to $\nabla$. If the Kodaira dimension $\kappa(M)$ is nonnegative, or more generally if there is a nontrivial section $\eta$ of $K_M^{ \ell}$ for some $0\neq\ell\in \mathbb{Z}$, then $\eta$ is parallel and it can never vanish, therefore $K_M^\ell$ is trivial.

(b) This result is due to Shiffman \cite[Lemma 2]{Sh}. In general we have the exact sequence
$$0\to H^1(M,\mathbb{R}/\mathbb{Z})\overset{\iota}{\to} \mathrm{Pic}(M)\overset{c_1^{\mathrm{BC}}}{\to} H^{1,1}_{\mathrm{BC}}(M,\mathbb{R}),$$
see \cite[(7)]{Ga}, and since the canonical bundle $K_{M}$ has $c_1^{\mathrm{BC}}(K_M)=-c_1^{\mathrm{BC}}(M)=0$,
there is a cocycle $\zeta\in H^1(M,\mathbb{R}/\mathbb{Z})$
such that $\iota(\zeta)=K_{M}$. On the other hand the universal coefficient theorem gives
$$H^1(M,\mathbb{R}/\mathbb{Z})\cong \mathrm{Hom}(H_1(M,\mathbb{Z}),\mathbb{R}/\mathbb{Z}).$$
The assumption that $b_1(M)=0$ implies that $H_1(M,\mathbb{Z})$ is a finite abelian group, so
$$\mathrm{Hom}(H_1(M,\mathbb{Z}),\mathbb{R}/\mathbb{Z})\cong H_1(M,\mathbb{Z})$$
and every element
in $H^1(M,\mathbb{R}/\mathbb{Z})$ is torsion. In particular, there is some $\ell\geq 1$ such that
$\ell\zeta=1$ and so $\iota(\ell\zeta)$ is the trivial bundle $\mathcal{O}_M$ in $\mathrm{Pic}(M)$
but at the same time $\iota(\ell\zeta)=K_M^{ \ell}$, which proves that
$K_M^{ \ell}\cong \mathcal{O}_M$
is holomorphically trivial.

(c) As we remarked in the introduction, this is a well-known consequence of the Calabi-Yau theorem and the Bogomolov-Calabi decomposition theorem \cite{Be, Bo, Ca}, and was
proved in this generality by Bogomolov \cite[Theorem 3]{Bo}, Fujiki \cite[Proposition 6.6]{Fu} and Lieberman \cite[Theorem 3.13]{Li}.

Indeed the decomposition theorem implies that a finite unramified cover $\ti{M}$ of $M$ splits as a product $T\times F$ where $T$ is a torus and $F$ is a compact K\"ahler manifold with $c_1(F)=0$ and $b_1(F)=0$. Obviously $K_T$ is trivial and $K_F$ is torsion by (b). It follows that $K_{\ti{M}}$ is torsion and therefore $K_M$ is torsion too.

(d) This is a theorem of Kodaira \cite[Theorem 38]{Ko}. For the reader's convenience, we give a sketch of proof (different from Kodaira's original one). By (a) and (c) if $K_M$ is not torsion then $\kappa(M)=-\infty$ and $M$ is not K\"ahler. By the Kodaira-Enriques classification \cite{BHPV} we have $b_1(M)=1$ and so $M$ is of class VII. In this case we have that $0=c_1^2(M)=-b_2(M)$, and it follows from results of Bogomolov \cite{Bo2}, Li-Yau-Zheng \cite{LYZ} and Teleman \cite{Te0} that $M$ is either a Hopf or an Inoue surface. But in both cases it is easy to see that $c_1^{\mathrm{BC}}(M)$ cannot be zero (see e.g. \cite[Remarks 4.2 and 4.3]{Te} and \cite[p.24]{TW2}, as well as Example \ref{exam2} for some special Hopf surfaces), and we are done.
Note that by the Kodaira-Enriques classification, compact complex surfaces with torsion canonical bundle are either K\"ahler or Kodaira surfaces.
\end{proof}

\begin{proof}[Proof of Theorem \ref{main}]
Since $M$ is in class $\mathcal{C}$, there is a modification $\mu:\ti{M}\to M$ with $\ti{M}$ a compact K\"ahler manifold. In fact, we may furthermore assume
that $\mu$ is a composition of blowups with smooth centers (see the proof of \cite[Theorem 2.2.16]{MM} for example).

Now if $\pi:X\to Y$ is the blowup of a complex manifold $Y$ along a complex submanifold $Z\subset Y$, with exceptional divisor $E\subset X$, then it is well-known that
$$K_X=\pi^*K_Y+(n-1-\dim Z)E.$$
From this it follows that in our case we have
$$K_{\ti{M}}=\mu^*K_M + \sum_E a_E E,$$
where the sum is over all $\mu$-exceptional divisors and the coefficients $a_E$ are nonnegative integers. If we denote by $F=\mu^*K_M^{-1}$, then $F$ is a holomorphic line bundle on $\ti{M}$ with $c_1^{\mathrm{BC}}(F)=0$, and $K_{\ti{M}}\otimes F$ is effective, i.e. $H^0(\ti{M}, K_{\ti{M}}\otimes F)\neq 0$.

The goal now is to show that this implies $\kappa(M)\geq 0$, since we can then apply Theorem \ref{old} (a), and finish the proof. We follow a strategy from \cite{CP}.

We need a recent result of Wang \cite{Wa} which extends a classical theorem of Simpson \cite{Si} from the projective to the K\"ahler case. To state this, define the cohomology jump loci
$$S^i_m=\{L\in \mathrm{Pic}^0(\ti{M})\ |\ h^i(\ti{M},L)\geq m\}.$$
The theorem of Wang \cite{Wa} (in the projective case due to Simpson \cite{Si}) says that for each value of $i,m$, the locus $S^i_m$ is a finite union of torsion translates of subtori of $\mathrm{Pic}^0(\ti{M})$.
In particular, we have
$$S:=S^n_1=\bigcup_{i=1}^N \{A_i+T_i\},$$
where each $A_i$ is a holomorphically torsion line bundle on $\ti{M}$, and each $T_i$ is a subtorus  of $\mathrm{Pic}^0(\ti{M})$. Thanks to Serre duality, we also have
$$S=\{L\in \mathrm{Pic}^0(\ti{M})\ |\ H^0(\ti{M},K_{\ti{M}}\otimes L^*)\neq 0\}.$$
If we let $S^*=\{L^*\ |\ L\in S\}$, then we also have that $S^*=\cup_{i=1}^N \{-A_i+T_i\}$.

By what we proved above, we have that $F\in S^*$, and so we can write $F=-A_i+F_1$ for some $A_i$ torsion and some $F_1\in T_i$. But we also have that
$F_1=F_2-A_j$ for some $A_j$ torsion and $F_2\in S$. We get $F=-A_i-A_j+F_2$, and so
$$H^0(\ti{M},K_{\ti{M}}\otimes F^*\otimes A_i^*\otimes A_j^*)\neq 0.$$
If $\ell\geq 1$ is such that $ A_i^\ell\otimes A_j^\ell$ is trivial, then
$$H^0(\ti{M},K_{\ti{M}}^\ell\otimes (F^*)^\ell)\neq 0.$$
But thanks to $H^0(\ti{M}, K_{\ti{M}}\otimes F)\neq 0$ we also have that
$$H^0(\ti{M},K_{\ti{M}}^\ell\otimes F^\ell)\neq 0,$$
and using the tensor product map
$$H^0(\ti{M},K_{\ti{M}}^\ell\otimes F^\ell)\otimes H^0(\ti{M},K_{\ti{M}}^\ell\otimes (F^*)^\ell)\to H^0(\ti{M},K_{\ti{M}}^{2\ell}),$$
we conclude that $ H^0(\ti{M},K_{\ti{M}}^{2\ell})\neq 0$. By the birational invariance of plurigenera, we also have $H^0(M,K_{M}^{2\ell})\neq 0$, which implies $\kappa(M)\geq 0$, as needed.
\end{proof}

\begin{proof}[Proof of Proposition \ref{deform}]
It is an elementary and well-known fact that for every compact complex manifold $b_1(M)\leq 2 h^{0,1}(M)$ with equality if and only if the $\de\db$-Lemma holds for $(1,1)$-forms (i.e. every $d$-exact $(1,1)$-form is $\de\db$-exact). Indeed, the map that associates to a real $1$-form its $(0,1)$ part induces an injection $H^1(M,\mathbb{R})\hookrightarrow H^1(M,\mathcal{O})$, whence $b_1(M)\leq 2 h^{0,1}(M)$, and then the map that associates to a $(0,1)$-form $a$ the $(1,1)$-form $\de a+\ov{\de a}$ induces an isomorphism
$$\frac{H^1(M,\mathcal{O})}{H^1(M,\mathbb{R})}\cong \frac{\{d\textrm{-exact real }(1,1)\textrm{-forms}\}}{\{\ddbar \psi, \psi\in C^\infty(M,\mathbb{R})\}}.$$
As a remark, a cohomological characterization of when the $\de\db$-Lemma holds for {\em all} $(p,q)$-forms was recently given in \cite{AT}.

Now the Betti number $b_1(M)$ is constant under small deformations, while the Hodge number $h^{0,1}(M)$ is upper semicontinuous. It follows that for $t$ sufficiently small we have
$$2h^{0,1}(M_t)\geq b_1(M_t)=b_1(M)=2h^{0,1}(M)\geq 2h^{0,1}(M_t),$$
and so the $\de\db$-Lemma holds for $(1,1)$-forms on $M_t$. The assumption that $c_1^{\mathrm{BC}}(M)=0$ implies that
$c_1(M)=0$ in $H^2(M,\mathbb{R})$ and so there is some $\ell\in \mathbb{N}$ such that
$\ell c_1(M)=0$ in $H^2(M,\mathbb{Z})$. This topological condition is preserved for small $t$,
so $c_1(M_t)=0$ in $H^2(M_t,\mathbb{R})$ and by the $\de\db$-Lemma we see that
$c_1^{\mathrm{BC}}(M_t)=0$.
\end{proof}

\begin{remark}
In fact, the proof of Proposition \ref{deform} shows that if $c_1^{\mathrm{BC}}(M)=0$ and for every $t\neq 0$ sufficiently small we have $b_1(M_t)=2h^{0,1}(M_t)$, then we have $c_1^{\mathrm{BC}}(M_t)=0$, even if $M$ did not satisfy the $\de\db$-Lemma for $(1,1)$ forms. An example where this phenomenon occurs is a complex parallelizable solvmanifold $M$ constructed by Nakamura \cite[Example III-(3b)]{Na}, which is also referred to as the parallelizable Nakamura manifold in \cite[Section 4]{AK}. This manifold has $b_1(M)=2, h^{0,1}(M)=3$ \cite{Na}, and therefore it does not satisfy the $\de\db$-Lemma. However, it admits small deformations $M_t$ with $h^{0,1}(M_t)=1$, which therefore do satisfy the $\de\db$-Lemma for $(1,1)$ forms (these are the ``case (1)'' deformations in \cite{AK}). It follows that $c_1^{\mathrm{BC}}(M_t)=0$.

Interestingly, this manifold has also other small deformations $N_t$ which do not satisfy the $\de\db$-Lemma for $(1,1)$ forms but still have $c_1^{\mathrm{BC}}(N_t)=0$. We will discuss them below in Example \ref{exam4}.
\end{remark}
\section{Examples}\label{sectex}

In this section we give some examples that elucidate the relations between the three conditions (1) $c_1^{\mathrm{BC}}(M)=0$, (2) $c_1(M)=0$ in $H^2(M,\mathbb{R})$ and (3) $K_M$ is holomorphically torsion.

\begin{example}\label{exam} Following Magnusson \cite{Mg}, we give examples of non-K\"ahler compact complex manifolds with vanishing first Bott-Chern class, whose canonical bundle is nevertheless not holomorphically torsion.

Let $X^n$ be a compact K\"ahler manifold with trivial canonical bundle and with an automorphism $f$ such that the induced automorphism on
$H^0(X, K_X)$ $\cong\mathbb{C}$ has infinite order. This implies that if $\Omega$ is a never-vanishing holomorphic $n$-form on $X$ then $f^*\Omega=\lambda\Omega$ with
$|\lambda|=1$ but $\lambda$ not a root of unity.
An example of such $X,f$ with $X$ a $2$-dimensional complex torus, due to Yoshihara \cite{Yo2}, is described in \cite[Example 6.4]{Ue2}.
There are also examples with $X$ a $K3$ surface, due to McMullen \cite{Mc}, and an explicit example with $X$ a $3$-dimensional complex torus, due to Iitaka \cite[Remark 14.6]{Ue}.

We will now describe Yoshihara's example in detail. Let $\alpha,\beta$ be the two roots of the equation
$$x^2-(1+i)x+1=0.$$
Then clearly $\alpha\beta=1$, but on the other hand $\alpha\ov{\beta}$ is not a root of unity. The minimal polynomial over $\mathbb{Q}$ of $\alpha$ (and $\ov{\beta}$) is
$$x^4-2x^3+4x^2-2x+1.$$
Let $\Lambda$ be the lattice in $\mathbb{C}^2$ spanned by the vectors
$(\alpha^j,\ov{\beta}^j)$, $j=0,\dots,3$, and $X=\mathbb{C}^2/\Lambda$. The automorphism of $\mathbb{C}^2$ given by multiplication by
\[\begin{pmatrix}
\alpha & 0  \\
0 & \ov{\beta}
\end{pmatrix}\]
descends to an automorphism $f$ of $X$, since $$(\alpha^4,\ov{\beta}^4)=2(\alpha^3,\ov{\beta}^3)-4(\alpha^2,\ov{\beta}^2)+2(\alpha,\ov{\beta})-(1,1).$$
The canonical holomorphic $2$ form $\Omega=dz^1\wedge dz^2$ on $X$ satisfies
$f^*\Omega=\alpha\ov{\beta}\Omega.$
The complex number $\alpha\ov{\beta}$ is not a root of unity, and hence the action of $f$ on $H^0(X,K_X)$ has infinite order.

Given any such $X,f$, one can then construct a holomorphic fiber bundle $M\to C$ with fiber $X$ where $C$ is an elliptic curve. In particular, if $X$ is
the above example then $M$ is a $3$-fold.
This construction has appeared many times in the literature
(see e.g. \cite[Remark 15.3]{Ue}, \cite[Example, p.15]{Bo}, \cite[p.248]{Fu}, \cite[Example 2.4]{AMN}, \cite[p.491]{Mj}), and is called ``suspension'' in \cite{AMN}. Namely, write $C=\mathbb{C}/(\mathbb{Z}\oplus\mathbb{Z}\tau)$, and define a holomorphic free $\mathbb{Z}^2$-action on $X\times\mathbb{C}$ by
$$(1,0)\cdot(x,z)=(x,z+1),\quad (0,1)\cdot(x,z)=(f(x),z+\tau).$$
The quotient is our manifold $M$, which fibers onto $C$. $M$ cannot be K\"ahler because the image of the monodromy map $\pi_1(C)\to \mathrm{Aut}H^n(X,\mathbb{R})$ contains an element of infinite order, violating \cite[Corollary 4.10]{Fu}.

Following \cite{Mg, Ue} we show that
$\kappa(M)=-\infty$ (so in particular $K_M$ cannot be torsion). Indeed, if we had a nontrivial section $s\in H^0(M,K_M^\ell)$ for some $\ell\geq 1,$ its pullback to $X\times\mathbb{C}$
would be of the form $F\cdot(\Omega\wedge dz)^\ell$ where $dz$ is the standard holomorphic $1$-form on $\mathbb{C}$ and $F$ is a holomorphic function on $X\times\mathbb{C}$. Since $X$ is compact, $F$ depends only on $z\in \mathbb{C}$. The pullback is invariant under the $\mathbb{Z}^2$-action, hence
$$F(z)\cdot(\Omega\wedge dz)^\ell=F(z+1)\cdot(\Omega\wedge dz)^\ell=\lambda^\ell F(z+\tau)\cdot(\Omega\wedge dz)^\ell,$$
and so $F(z)=F(z+1)=\lambda^\ell F(z+\tau)$. Therefore
$|F(z)|=|F(z+1)|=|F(z+\tau)|$ and so $F$ is constant by the maximum modulus principle. This implies that $\lambda^\ell=1$ which contradicts the fact that $\lambda$ is not a root of unity.

Finally, Magnusson \cite{Mg} shows that $c_1^{\mathrm{BC}}(M)=0$. To see this, fix a Ricci-flat K\"ahler metric $\omega$ on $X$. We have that $f^*\omega^n=\omega^n,$
because both $\omega^n$ and $f^*\omega^n$ are Ricci-flat volume forms with the same integral. Consider now
the volume form $\omega^n\wedge\omega_E$ on $X\times\mathbb{C}$, where $\omega_E$ is the standard Euclidean metric on $\mathbb{C}$. It is Ricci-flat and invariant under the $\mathbb{Z}^2$-action, and therefore it descends to a Ricci-flat volume form on
$M$, i.e. a flat Hermitian metric on $K_M$.
\end{example}

\begin{example}\label{exam4}
We now discuss other examples of non-K\"ahler compact complex manifolds with vanishing first Bott-Chern class and with canonical bundle not holomorphically torsion.

The complex parallelizable solvmanifold $M$ constructed by Nakamura \cite[Example III-(3b)]{Na} (see also \cite[Section 4]{AK}), has a family of small deformations $N_t$ which
have $b_1(N_t)=h^{0,1}(N_t)=2$ (``case (2)'' in \cite{AK}) and therefore do not satisfy the $\de\db$-Lemma. They also have Kodaira dimension $\kappa(N_t)=-\infty$ \cite{Na}. Nevertheless, these manifolds also satisfy $c_1^{\mathrm{BC}}(N_t)=0$. This can be seen by using some computations in \cite[Table 4]{AK}: we have that $M=\mathbb{C}^3/\Gamma$ for a certain discrete subgroup $\Gamma$, and using the standard coordinates on $\mathbb{C}^3$ we can write the complex structure on $N_t$ as given by the infinitesimal deformation vector $t\de_1\otimes e^{z_1}d\ov{z}_3\in H^{0,1}(M,TM),$ where $t$ is a small complex number.
The following $1$-forms give well-defined linearly independent $(1,0)$-forms on $N_t$:
$$\theta_1=dz_1-te^{z_1}d\ov{z}_3,\quad \theta_2=e^{-z_1}dz_2, \quad \theta_3=e^{z_1}dz_3.$$
Then $\omega=\mn \theta_1\wedge\ov{\theta_1}+\mn \theta_2\wedge\ov{\theta_2}+\mn \theta_3\wedge\ov{\theta_3}$ defines a Hermitian metric on $N_t$. We have
$$\omega^3=6 (\mn)^3 dz_1\wedge d\ov{z}_1\wedge dz_2\wedge d\ov{z}_2\wedge dz_3\wedge d\ov{z}_3,$$
which in local holomorphic coordinates on $N_t$ is a constant times the Euclidean volume form,
and hence $\Ric(\omega)=0$ on $N_t$.
\end{example}

\begin{example}\label{exam2}
We now give some examples of non-K\"ahler compact complex manifolds with $c_1(M)=0$ in $H^2(M,\mathbb{R})$ but with
$c_1^{\mathrm{BC}}(M)\neq 0$. We start with the following simple observation: if $M$ is a compact complex manifold that admits a Hermitian metric $\omega$ with $\Ric(\omega)\geq 0$ but not identically zero, then $c_1^{\mathrm{BC}}(M)\neq 0$. Indeed if we had $c_1^{\mathrm{BC}}(M)=0$ then there would be a smooth function $F$ with $\Ric(\omega)=\ddbar F\geq 0,$ which implies that $F$ must be constant and so $\Ric(\omega)=0$ which we assumed is not the case.

We apply this observation to the Hopf manifold $M=(\mathbb{C}^n\backslash\{0\})/\sim$, $n\geq 2$, where we identify
$(z_1\dots,z_n)\sim(\alpha_1z_1,\dots,\alpha_nz_n)$, and the nonzero complex numbers $\alpha_j$ all have the same modulus which is different from $1$. The complex manifold $M$ is diffeomorphic to $S^1\times S^{2n-1}$, so that $b_2(M)=0$. We consider the Hermitian metric on $M$ given by
$$\omega=\frac{\delta_{ij}}{|z|^2}\mn dz_i\wedge d\ov{z}_j.$$
A simple calculation (see e.g. \cite[p.28]{TW2}) shows that
$$\Ric(\omega)=\frac{n}{|z|^2}\left(\delta_{ij}-\frac{\ov{z}_i z_j}{|z|^2}\right)\mn dz_i\wedge d\ov{z}_j,$$
which is clearly not identically zero and is semipositive definite by the Cauchy-Schwarz inequality.
It follows that $c_1(M)=0$ in $H^2(M,\mathbb{R})$, yet $c_1^{\mathrm{BC}}(M)\neq 0$.
\end{example}

\begin{example}\label{exam5}
The following example was suggested by the referee. This is again a non-K\"ahler compact complex manifold with $c_1(M)=0$ in $H^2(M,\mathbb{R})$ but with $c_1^{\mathrm{BC}}(M)\neq 0$, obtained as a principal torus
bundle over a compact Riemann surface $\Sigma$ of genus $g\geq 2$. Let $T=\mathbb{C}^n/\Lambda$ be an $n$-dimensional complex torus, and $\pi:M\to\Sigma$ be any topologically nontrivial principal $T$-bundle over $\Sigma$.
Then the defining cocycle of the bundle determines a characteristic class $c\in H^2(\Sigma,\mathbb{Z})\otimes\Lambda$, which can be viewed as a map $\delta:H^1(T,\mathbb{Z})\to H^2(\Sigma,\mathbb{Z})$, and this vanishes precisely
when $M$ is topologically trivial. Therefore in our case $\delta$ is surjective, and furthermore $M$ is non-K\"ahler (see e.g. \cite[Theorem 1.7]{Ho}). In general we have that $K_M=\pi^*K_\Sigma$, and so $\kappa(M)=1$, and $c_1(M)=\pi^*c_1(\Sigma)$ vanishes in $H^2(M,\mathbb{Z})$ iff $c_1(\Sigma)$ is in the image of $\delta$. In our case this holds, so $c_1(M)=0$ in $H^2(M,\mathbb{R})$. If we had $c_1^{\mathrm{BC}}(M)=0$, then
we would get a contradiction from Theorem \ref{old} (a), because $\kappa(M)=1>0$ and $K_M$ is not holomorphically torsion.
\end{example}

\begin{example}\label{exam3}
Here we consider a compact complex manifold $M$ diffeomorphic to the six-sphere $S^6$, assuming one exists. Of course it is a well-known open problem to determine whether such a manifold $M$ exists.

Obviously we have $c_1(M)=0$ in $H^2(M,\mathbb{R})$, and we now show that $c_1^{\mathrm{BC}}(M)\neq 0$. Indeed, if we had $c_1^{\mathrm{BC}}(M)=0$ then we would have that $K_M$ is holomorphically torsion thanks to Theorem \ref{old} (b). However, the exponential exact sequence together with $H^1(M,\mathbb{Z})=H^3(M,\mathbb{Z})=0$ imply that $\mathrm{Pic}(M)\cong H^1(M,\mathcal{O}_M)$, which has no torsion. Therefore $K_M$ is trivial, and there is a never-vanishing holomorphic $3$-form $\Omega$. This form is clearly $d$-closed, and hence $d$-exact, $\Omega=d\beta$. Then
$$0<(\mn)^{n^2}\int_M \Omega\wedge\ov{\Omega}=(\mn)^{n^2}\int_M d(\beta\wedge d\ov{\beta})=0,$$
a contradiction. The last part of this argument comes from \cite{Gr}.
\end{example}

\section{Canonical metrics on non-K\"ahler Calabi-Yau manifolds}\label{sectbal}
A K\"ahler Calabi-Yau manifold admits Ricci-flat K\"ahler metrics, exactly one in each K\"ahler class \cite{Ya1}. These canonical metrics have proved extremely useful in the study of the geometry of Calabi-Yau manifolds. It is natural to ask whether an analog of such metrics exists on non-K\"ahler Calabi-Yau manifolds.

We do not have a satisfactory general answer to this question, but there are several possible approaches. First of all, it is clear from Theorem \ref{metrics} that Chern-Ricci flat metrics (which always exist on manifolds with vanishing first Bott-Chern class) are not canonical in any reasonable sense: there are simply too many of them, because every Hermitian metric is conformal to a Chern-Ricci flat metric. However, if we restrict to Hermitian metrics which satisfy additional hypotheses, there is some hope to construct suitable canonical metrics.

For example, let $M$ be a compact complex manifold with $c_1^{\mathrm{BC}}(M)= 0$, which admits a {\em balanced} Hermitian metric $\omega$, which by definition satisfies
$d(\omega^{n-1})=0$, see \cite{Mi}. Is it possible to find another balanced metric on $M$ which is Chern-Ricci flat? For emphasis, we state this as a conjecture (which is part of the folklore of this subject, see e.g. \cite{F, FX, FWW, Po}).

\begin{conj}\label{balanced}
Let $M$ be a compact complex manifold with $c_1^{\mathrm{BC}}(M)= 0$ and with a balanced metric $\omega$. Then there is a balanced metric $\ti{\omega}$ with $[\ti{\omega}^{n-1}]=[\omega^{n-1}]$ in $H^{2n-2}(M,\mathbb{R})$ with $\Ric(\ti{\omega})=0$.
\end{conj}

There are many examples of such manifolds. Indeed, any manifold in class $\mathcal{C}$ is balanced by \cite{AB}, and there are many with $c_1(M)=0$ (see the discussion in the Introduction before Theorem \ref{main}). There are also examples not in class $\mathcal{C}$, which are diffeomorphic to connected sums $\sharp_{i=1}^N (S^3\times S^3), N>1$: these manifolds have trivial canonical bundle by \cite{Fr, LT} and admit balanced metrics by \cite{FLY}, but are not in class $\mathcal{C}$ since they have vanishing second Betti number. More examples of non-K\"ahler compact complex manifolds with trivial canonical bundle and which admit balanced metrics were constructed in \cite{GP, FY}.

It is interesting to note that if $\omega$ is a balanced metric with $\Ric(\omega)=0$, then the Bismut connection \cite{Bi} of $\omega$ (also known as the connection with skew symmetric torsion) has vanishing Ricci curvature as well, which means that its restricted holonomy is also contained in $SU(n)$. This condition has been much studied in the mathematical physics literature \cite{AI, FG, Fr3, Gr, GIP, IP, Ph}, and Conjecture \ref{balanced} would provide many more examples of such special metrics.

An approach to Conjecture \ref{balanced} was proposed in \cite{FWW}. Given a balanced metric $\omega$ on $M$, the condition $c_1^{\mathrm{BC}}(M)= 0$ implies that $\Ric(\omega)=\ddbar F$ for some smooth function $F$. Then we seek a new balanced metric $\ti{\omega}$ such that $\ti{\omega}^{n-1}=\omega^{n-1}+\ddbar (u\omega_0^{n-2})$ for some smooth function $u$ and some Hermitian metric $\omega_0$ on $M$, with $\ti{\omega}$ solving the Monge-Amp\`ere equation
\begin{equation}\label{balma}
\ti{\omega}^n=e^{F+b}\omega^n,
\end{equation}
for some constant $b$. This is called a ``form-type Calabi-Yau equation'' in \cite{FWW}. Regarding the solvability of \eqref{balma} we have the following conjecture:

\begin{conj}\label{solvebalanced}
Let $M$ be a compact complex manifold with a balanced metric $\omega$, a Hermitian metric $\omega_0$, a smooth function $F$.
Then there are a constant $b$ and a balanced metric $\ti{\omega}$ with $\ti{\omega}^{n-1}=\omega^{n-1}+\ddbar (u\omega_0^{n-2})$ for some smooth function $u$, such that
\eqref{balma} holds.
\end{conj}

Conjecture \ref{solvebalanced} implies Conjecture \ref{balanced}, since with the above choice of $F$ we can apply $\ddbar\log$ to \eqref{balma} and see that $\Ric(\ti{\omega})=\Ric(\omega)-\ddbar F=0$, and clearly $[\ti{\omega}^{n-1}]=[\omega^{n-1}]$ in $H^{2n-2}(M,\mathbb{R})$. We have the following theorem:

\begin{theorem}\label{tw}
Conjecture \ref{solvebalanced} holds if the metric $\omega_0$ is K\"ahler. Therefore Conjecture \ref{balanced} holds when $M$ is K\"ahler.
\end{theorem}
This theorem was proved by Weinkove and the author in \cite{TW4}. Indeed, since $\omega_0$ is K\"ahler, we see that
$$\omega^{n-1}+\ddbar (u\omega_0^{n-2})=\omega^{n-1}+\ddbar u\wedge\omega_0^{n-2},$$
and so \eqref{balma} becomes the ``Monge-Amp\`ere equation for $(n-1)$-plurisubharmonic functions'' solved in \cite{TW4}.
Note that even though in this case the manifold $M$ is K\"ahler, the Chern-Ricci flat balanced metrics $\ti{\omega}$ that we get are usually not K\"ahler (if $n\geq 3$).
Note also that when $M$ is K\"ahler, Conjecture \ref{balanced} does not trivially follow from Yau's Theorem \cite{Ya1}, since in general given a balanced metric $\omega$ there is no K\"ahler metric $\ti{\omega}$ with $[\ti{\omega}^{n-1}]=[\omega^{n-1}]$ (see \cite{FX}).

If instead of $d\omega_0=0$ we assume the astheno-K\"ahler condition $\de\db (\omega_0^{n-2})=0$ of Jost-Yau \cite{JY}, then we
have that
$$\omega^{n-1}+\ddbar (u\omega_0^{n-2})=\omega^{n-1}+\ddbar u\wedge\omega_0^{n-2}+2\mathrm{Re}(\mn \de u \wedge\db(\omega_0^{n-2})),$$
and \eqref{balma} falls into the class of equations studied in \cite{TW5, Po2} (the only difference is the factor of $2$ in front of $\mathrm{Re}(\mn \de u \wedge\db(\omega_0^{n-2})$,
which does not affect any of the results in \cite{TW5}). In particular, if $M$ admits astheno-K\"ahler metrics, then Conjectures \ref{balanced} and \ref{solvebalanced} are reduced to proving a suitable second order estimate for the solution $u$ (see \cite[Theorem 1.7]{TW5}).

Finally, let us also mention that another candidate for a class of special metrics on certain non-K\"ahler Calabi-Yau manifolds are solutions of Strominger's system \cite{St}. In general, solutions of such system are extremely hard to construct, see \cite{FY, LYa,Ph, TY} and references therein for more about this system.

\end{document}